\documentclass{amsart}
\usepackage{comment}
\usepackage{tgtermes}
\usepackage{newtxmath}

\usepackage{amsmath}
\usepackage{mathtools}
\usepackage[usenames,dvipsnames]{color}
\usepackage[normalem]{ulem}
\usepackage{url}
\usepackage{graphicx}
\usepackage{float, caption}
\usepackage{enumerate}
\usepackage{hyperref}  
\usepackage{thmtools}
\usepackage{thm-restate}
\usepackage{xcolor} 
\definecolor{royalpurple}{RGB}{120, 81, 169} 

\hypersetup{%
  bookmarksnumbered=true,%
  colorlinks=true,%
  linkcolor=royalpurple,
  citecolor=blue,
  filecolor=blue,%
  menucolor=blue,%
  urlcolor=blue
}

\def\makeautorefname#1#2{\expandafter\def\csname#1autorefname\endcsname{#2}}
\def\equationautorefname~#1ll{(#1)\null}
\makeautorefname{footnote}{footnote}%
\makeautorefname{item}{item}%
\makeautorefname{figure}{Figure}%
\makeautorefname{table}{Table}%
\makeautorefname{part}{Part}%
\makeautorefname{appendix}{Appendix}%
\makeautorefname{chapter}{Chapter}%
\makeautorefname{section}{Section}%
\makeautorefname{subsection}{Section}%
\makeautorefname{subsubsection}{Section}%
\makeautorefname{theorem}{Theorem}%
\makeautorefname{thm}{Theorem}%
\makeautorefname{cor}{Corollary}%
\makeautorefname{lem}{Lemma}%
\makeautorefname{prop}{Proposition}%
\makeautorefname{pro}{Property}
\makeautorefname{conj}{Conjecture}%
\makeautorefname{defn}{Definition}%
\makeautorefname{notn}{Notation}
\makeautorefname{notns}{Notations}
\makeautorefname{rem}{Remark}%
\makeautorefname{quest}{Question}%
\makeautorefname{exmp}{Example}%
\makeautorefname{ax}{Axiom}%
\makeautorefname{claim}{Claim}%
\makeautorefname{ass}{Assumption}%
\makeautorefname{asss}{Assumptions}%
\makeautorefname{con}{Construction}%
\makeautorefname{prob}{Problem}%
\makeautorefname{warn}{Warning}%
\makeautorefname{obs}{Observation}%
\makeautorefname{conv}{Convention}%

\newtheorem{thm}{Theorem}
\newtheorem{cor}[thm]{Corollary}
\newtheorem{prop}[thm]{Proposition}
\newtheorem{lem}[thm]{Lemma}

\newtheorem*{thm*}{Theorem}
\newtheorem*{lem*}{Lemma}
\newtheorem*{cor*}{Corollary}
\newtheorem*{prop*}{Proposition}

\theoremstyle{definition}

\newtheorem{con}[thm]{Construction}

\newtheorem{quest}[thm]{Question}
\newtheorem{rem}[thm]{Remark}

\DeclareMathOperator{\ch}{Conv}

\DeclareMathOperator{\intr}{int}

\newcommand{\Z}{\mathbf{Z}}

\newcommand{\R}{\mathbf{R}}

\newcommand{\lb}{\langle}
\newcommand{\rb}{\rangle}

\title{Carousel theorems for compact sets and homothets}

\author{Yiming Song}
\thanks{Department of Mathematics, Duke University, Durham, NC 27708, USA. Email: \texttt{y.song@duke.edu}.}


\begin{document}
\setcounter{tocdepth}{1}

\begin{abstract}
  We prove that if $A_0$ and $A_1$ are compact sets contained in a convex $n$-gon with vertices $G_1, \dots, G_n$, and $2 \lceil \frac{n}{2} \rceil$ is strictly greater than the number of common supporting lines of $A_0$ and $A_1$, then there exist $i\in \{0,1\}$ and $j\in \{1, \dots, n\}$ such that $A_i$ is contained in the convex hull of $A_{1-i}$ and $\{G_1,\dots,G_n\}\setminus\{G_j\}$. This generalizes and recovers results of Adaricheva--Bolat and Cz{\'e}dli--Kurusa concerning disks and homothetic sets. We construct examples to prove that the bound is sharp. We also construct a family of convex geometries not representable by positive homothets of any fixed planar convex body.
\end{abstract}
\maketitle

\section{Introduction} 
In 2016, K. Adaricheva and M. Bolat proved the following result:

\begin{thm*}[{\cite[Theorem 3.1]{ADARICHEVA2019726}}]
If $A_0$ and $A_1$ are closed disks in $\R^2$ and $G$ is a triangle with vertices $G_0, G_1, G_2$ such that $A_0, A_1 \subset G$, then there exist $i \in \{0,1\}$ and $j \in \{0,1,2\}$ such that $A_i \subset \ch(A_{1-i}, \cup_{k\neq j} \{G_k\})$.
\end{thm*}

\noindent Throughout the paper, let $\ch$ denote the convex hull operator and take $\subset$ to mean (not necessarily strict) inclusion. This theorem was later given a more succinct proof by G. Cz{\'e}dli in \cite{czedliCounterexamples} and extended by G. Cz{\'e}dli and {\'A}. Kurusa to the case where $A_0$ is obtained from $A_1$ via a positive homothety in \cite{Czedli2019}.

Our first result generalizes this theorem to the case where $A_0$ and $A_1$ are arbitrary compact sets in the plane. The relevant parameter turns out to be the number of common supporting lines of $A_0$ and $A_1$. A pair of distinct disks or distinct positively homothetic convex bodies has at most two, so the aforementioned theorems follow from:

\begin{restatable}{thm}{mainthm}\label{thm:carousel_points} Let $A_0$ and $A_1$ be compact subsets of $\R^2$ contained in $G$, a convex $n$-gon with vertices $G_1, \dots, G_n$. Let $s$ be the number of common supporting lines of $A_0$ and $A_1$. Suppose that
  \[s<
\begin{cases}
n, & \text{if $n$ is even},\\
n+1, & \text{if $n$ is odd}.
\end{cases}
\]
Then there exist $i\in\{0,1\}$ and $j \in \{1,\dots, n\}$ such that 
\[A_i \subset \ch(A_{1-i}, \cup_{k\neq j} \{G_k\}).\]
\end{restatable}
\noindent We construct explicit examples in Section \ref{subsection:bound_sharp} to show that this bound is sharp. For several natural families of sets, the number of common supporting lines can be explicitly bounded. For example, when $\partial\ch(A_0)$ and $\partial\ch(A_1)$ are algebraic curves (see Corollary \ref{cor:algebraic}). The proof of Theorem~\ref{thm:carousel_points}, which is convex-analytic and uses the machinery of support functions, also extends to the case where $A_0$ and $A_1$ have an infinite set of common supporting lines (see Corollary~\ref{cor:weak_carousel_inf_supplines}). This answers a question posed to the author by G. Cz{\'e}dli.

Now for terminology. When the conclusion of Theorem~\ref{thm:carousel_points} holds, we say that the ordered $(n+2)$-tuple $(A_0, A_1, G_1,\dots, G_n)$ satisfies the \emph{weak carousel rule}. The word ``carousel'' is used because the convex hulls considered in the theorem look somewhat like carousels. The term was introduced in \cite{Adaricheva2010}, which also contains figures of said carousel. The word ``weak'' indicates that the containment is only required to hold for a single $i\in \{0,1\}$. In this language, the theorem of Adaricheva--Bolat asserts that the weak carousel rule is satisfied by two disks contained in a triangle. 

The weak carousel rule, while purely geometric, has applications to combinatorics, in particular to the representation of \emph{convex geometries}, which are set systems equipped with a type of closure operator (see Section~\ref{section:convexgeometries} for key definitions). Roughly speaking, a \emph{representation} by a class $\mathcal{C}$ of convex bodies in $\R^d$ assigns an element of $\mathcal{C}$ to each element of a finite convex geometry so that its closure operator is realized by ordinary convex-hull containment. Adaricheva and Bolat show in \cite{ADARICHEVA2019726} that the weak carousel rule yields obstructions to representing convex geometries by the class of disks in the plane.

Representations of finite convex geometries by configurations of convex bodies in Euclidean space have been studied extensively. A seminal result of Kashiwabara, Nakamura, and Okamoto \cite{KASHIWABARA2005129} shows that every finite convex geometry embeds into a point configuration in sufficiently high dimension, while Cz{\'e}dli \cite{CZEDLI201461} shows that planar disks represent a broader class of convex geometries than planar point configurations. Richter and Rogers \cite{RICHTER20171059} represent convex geometries of convex dimension $n$ by convex $n$-gons in $\R^2$ (see \cite{edelmanjamison} for definitions), and Cz{\'e}dli and Kincses \cite[Theorem 1.8]{CzedliKincses2017} extend such results to differentiable convex curves arbitrarily close to circles. Kincses \cite{Kincses2017} further shows that no uniform bound on the number of pairwise common supporting lines suffices to represent all finite convex geometries. 

Our next result complements the result of Kincses and extends that of Adaricheva--Bolat. We show that the weak carousel rule obstructs the representation of a certain family of convex geometries by the class $\mathcal{H}(K)$ of positive homothets of a given convex \emph{body} (compact with nonempty interior) $K\subset \R^2$. To obtain such an obstruction, we must show that the weak carousel rule holds when not only $A_0$ and $A_1$, but also $G_1,\dots, G_n$, belong to $\mathcal{H}(K)$:
\begin{thm}\label{thm:nonrep_exmp}
  Let $n\geq 3$. Fix a convex body $K\subset \R^2$. Let $\mathcal{H}(K) = \{t + \lambda K: t\in \R^2, \lambda \in \R_{>0}\}$. For any $A_0, A_1, G_1, \dots, G_n \in \mathcal{H}(K)$ with $A_0, A_1 \subset \ch(\cup_j G_j)$, the weak carousel rule holds, i.e.  there exist $i\in\{0,1\}$ and $j \in \{1,\dots, n\}$ such that 
\[A_i \subset \ch(A_{1-i}, \cup_{k\neq j} G_k).\]
\end{thm}
Note that while two positively homothetic bodies have at most two common supporting lines, Theorem~\ref{thm:carousel_points} is not enough, since the $G_j$ are points in its setup. The proof of Theorem~\ref{thm:nonrep_exmp} relies again on properties of support functions; they behave nicely under positive homothety. When $K$ is a disk, we recover \cite[Theorem 5.2]{ADARICHEVA2019726}, which is the weak carousel rule for two disks contained in the convex hull of three other disks. We then generalize a construction from \cite{ADARICHEVA2019726} to obtain convex geometries $(U_n, \tau_n)$ for $n\geq 3$ that are not representable by positive homothets of a fixed convex body (Corollary~\ref{cor:nonrep}).

\subsection*{Organization} In Section~\ref{section:prelims}, we recall the necessary background in support functions and convex analysis. In Section~\ref{section:weakcarousel}, we prove Theorem~\ref{thm:carousel_points}, construct explicit examples to show sharpness, and list several corollaries. In Section~\ref{section:convexgeometries}, we prove Theorem~\ref{thm:nonrep_exmp} and construct a family of convex geometries not representable by positive homothets of a fixed convex body. In Section~\ref{section:questions} we conclude with a few open questions.

\section{Common supporting hyperplanes and support functions}\label{section:prelims}

Let $A \subset \R^d$ be a compact set and let $l$ be a hyperplane in $\R^d$. We say that $l$ \textit{supports} $A$, or is a \textit{supporting hyperplane} of $A$, if $l\cap A$ is nonempty and $A$ lies in one of the closed halfspaces bounded by $l$. Given a collection $\mathcal{A} = \{A_i\}$ of compact sets in $\R^d$, we say that $l$ is a \textit{common supporting hyperplane} of $\mathcal{A}$ if $l$ is a supporting hyperplane for each $A_i$ and all $A_i$ lie in the same closed halfspace bounded by $l$.

Given a nonempty compact\footnote{For every nonempty compact set $A$, the support function is finite and continuous; for an unbounded closed convex set it may take the value $+\infty$.} subset $A \subset \R^d$, it is useful to consider its \textit{support function} $h_A: \R^d \to \R$, given by $h_A(\eta) := \sup_{a \in A} \lb \eta, a\rb$,
where $\lb -, -\rb$ denotes the standard inner product. The support function $h_A$ is finite, continuous, and positively homogeneous, so we can usually just consider its restriction to $S^{d-1}\subset \R^d$. For $\eta \in S^{d-1}$ and nonempty compact sets $A_0, A_1$, call $\eta$ a \emph{common supporting direction} of $A_0, A_1$ if $h_{A_0}(\eta) = h_{A_1}(\eta)$. 

We list several useful properties of support functions in the following lemma. These properties are standard; see, for example,  \cite[\S1.7.1]{Schneider_2013}, which uses an equivalent analytic definition for supporting hyperplanes.

\begin{lem}\label{lem:support_function_properties} Let $A, B\subset \R^d$ be nonempty, convex, compact sets, and $G_1,\dots, G_n\subset \R^d$ be nonempty and compact.
\begin{enumerate}[(a)]
  \item For a fixed unit vector $\eta \in S^{d-1}$, 
  \[l_\eta :=  \{x \in \R^d: \lb x, \eta\rb  = h_{G_1}(\eta)\}\]
  is a supporting hyperplane of $G_1$. Every supporting hyperplane of $G_1$ takes this form.
  \item For a fixed unit vector $\eta \in S^{d-1}$ and constant $c \in \R$, the hyperplane $\{x \in \R^d: \lb \eta, x\rb =  c\}$ supports $G_1, \dots, G_n$ in the half-space $\{\lb \eta, x\rb \leq c\}$ if and only if $h_{G_i}(\eta) = c$ for all $i \in \{1, \dots, n\}$. 
  \item $A\subset B$ if and only if $h_A(\eta)\leq h_B(\eta)$ for all $\eta \in S^{d-1}$.
  \item $h_{\ch(\cup_i G_i)}(\eta) = \max_i h_{G_i}(\eta)$.
\end{enumerate}
\end{lem}

\begin{rem}\label{rem:convex}
Notice that for compact sets $U, V_1, \dots, V_n \subset \R^d$, we have, by the idempotency and monotonicity of the convex hull operator, that $U\subset \ch(\cup_i V_i)$ if and only if $\ch(U)\subset \ch(\cup_i\ch V_i)$. This is because $\ch(\cup_i \ch(V_i)) = \ch(\cup_i V_i)$. Moreover, if $U$ is nonempty then $h_U = h_{\ch(U)}$, and a hyperplane supports $U$ if and only if it supports $\ch(U)$. Finally, homotheties and translations commute with taking convex hulls. Therefore, in proving several results in this paper, we may replace the compact sets involved by their convex hulls.
\end{rem}

\begin{lem}\label{lem:carousel_general}
Let $A_0, A_1$ be compact subsets of $\R^d$ contained in $G = \ch(G_1, \dots, G_n)$ where each $G_k$ is a nonempty compact subset of $\R^d$. For each $j \in \{1, \dots, n\}$ define $\Lambda_j := \{\eta \in S^{d-1}: h_G(\eta) = h_{G_j}(\eta)\}$. Suppose either 
\begin{enumerate}
  \item $A_0 = \varnothing$ or $A_1 = \varnothing$, or
  \item there exists $j$ such that $h_{A_0}-h_{A_1}$ has constant sign (nonpositive or nonnegative) on $\Lambda_j$.
\end{enumerate}
Then $(A_0, A_1, G_1, \dots, G_n)$ satisfies the weak carousel rule, i.e. there exist $i \in \{0,1\}$ and $j\in \{1, \dots, n \}$ such that $A_i \subset \ch(A_{1-i}, \cup_{k\neq j} G_k)$.
\begin{proof}
By Remark~\ref{rem:convex} we may replace each $A_i$ and $G_j$ by its convex hull. The result is trivial if either $A_0$ or $A_1$ is empty, so suppose otherwise. Let $j$ be as in the hypothesis of the lemma. Then, after swapping $A_0$ and $A_1$ as necessary, $h_{A_0}(\eta) \leq h_{A_1}(\eta)$ for $\eta \in \Lambda_j$. It follows that 
  \[
  h_{A_0}(\eta) \leq h_{A_1}(\eta) \leq \max(h_{A_1}(\eta), \max_{k\neq j}h_{G_k}(\eta)) = h_{\ch(A_1, \cup_{k\neq j}G_k)}(\eta)
  \]
  for $\eta \in \Lambda_j$. Here we have used Lemma~\ref{lem:support_function_properties}(c,d). 
  
  Now suppose $\eta \notin \Lambda_j$. By Lemma~\ref{lem:support_function_properties}(d), $h_{G}(\eta)= \max_k h_{G_k}(\eta)$, so there exists $k\neq j$ such that $h_{G}(\eta)= \max_k h_{G_k}(\eta)$. Since $A_0\subset G$, we obtain 
  \[h_{A_0}(\eta) \leq h_G(\eta) = h_{G_k}(\eta) \leq \max(h_{A_1}(\eta), \max_{m\neq j}h_{G_m}(\eta)) = h_{\ch(A_1, \cup_{m\neq j}G_m)}(\eta).\]
  Hence $h_{A_0}(\eta) \leq h_{\ch(A_1,\cup_{k\neq j}G_k)}(\eta)$ for all $\eta \in S^{d-1}$, so by Lemma~\ref{lem:support_function_properties}(c), we conclude that $A_0 \subset \ch(A_1, \cup_{k\neq j} G_k)$.
\end{proof}
\end{lem}

\section{The weak carousel rule}\label{section:weakcarousel}
In this section, we prove Theorem~\ref{thm:carousel_points}, construct examples to demonstrate its bound is sharp, and list a few of its corollaries. 
\subsection{Proof of Theorem~\ref{thm:carousel_points} }
\begin{proof}[Proof of Theorem~\ref{thm:carousel_points}]
  In this case the $G_j$ are just points. Again, by Remark~\ref{rem:convex} we may assume $A_0$ and $A_1$ are convex. We may also assume that both $A_0$ and $A_1$ are nonempty. By definition, the zeros of $h_{A_0}-h_{A_1}$ correspond to the common supporting directions of $A_0$ and $A_1$. First suppose $A_0\cup A_1$ is contained in a line. Let $I$ be the interval $\ch(A_0, A_1)$. If both ends of the interval belong to the same set $A_0$ or $A_1$, the result is immediate. Otherwise, suppose $a \in A_0$ and $b \in A_1$ are the two endpoints. Since the collection $\ch(b, \cup_{k\neq j}G_k)$ for $j \in \{1, \dots, n\}$ covers $G$, we get 
  \[\ch(A_0, A_1) = \ch(a,b) \subset \ch(b, \cup_{k\neq j}G_k)\subset \ch(A_1, \cup_{k\neq j}G_k)\]
  for some $j$ as desired.
  
  If $A_0 \cup A_1$ is not contained in a line, then each common supporting line corresponds to exactly one common supporting direction, hence $h:= h_{A_0}-h_{A_1}$ has $s$ zeros. It now suffices by Lemma~\ref{lem:carousel_general} to find an index $j \in \{1, \dots, n\}$ such that $h$ has constant sign on $\Lambda_j$. Here $\Lambda_j = \{\eta \in S^1: h_G(\eta) = h_{G_j}(\eta)\}$ is the normal cone of $G$ at the vertex $G_j$, intersected with $S^1$. These cover $S^1$ and have pairwise disjoint interiors.

  To find the index $j$, we proceed based on the parity of $n$. First suppose $n$ is even. The assumption is that $s<n$, so by the pigeonhole principle, there exists some $\Lambda_j$ on which $h$ has constant sign. Now suppose $n$ is odd and $s< n+1$. On $S^1$, the continuous function $h$ has to change sign an even number of times, so it can change sign at most $n-1$ times. Then by the pigeonhole principle again, there is some $\Lambda_j$ on which $h$ does not change sign. We conclude via Lemma~\ref{lem:carousel_general}.
\end{proof}

We note that in the proof of Theorem~\ref{thm:carousel_points} above, one only relies on the fact that if $A_0$ and $A_1$ are not collinear, then having $s$ common supporting lines implies that the function $h:= h_{A_0}-h_{A_1}$ changes sign at most $s$ times. So the proof proceeds identically if we suppose that $h^{-1}(0)$ has $\tilde{s}$ connected components, since this also implies that $h$ changes sign at most $\tilde{s}$ times:
\begin{cor}\label{cor:weak_carousel_inf_supplines}
  Let $A_0$, $A_1$, $G$, and $n$ be as in Theorem~\ref{thm:carousel_points}. Suppose $A_0$ and $A_1$ are nonempty, and let $\tilde{s}$ be the number of connected components of the zero set of $h_{A_0}-h_{A_1}$. Suppose 
  \[\tilde{s}<
    \begin{cases}
    n, & \text{if $n$ is even},\\
    n+1, & \text{if $n$ is odd}.
    \end{cases}
  \]
  Then $(A_0, A_1, G_1,\dots, G_n)$ satisfies the weak carousel rule. \hfill \qed
\end{cor}

\subsection{The bound is sharp}\label{subsection:bound_sharp}
Given even $n\geq 4$ (resp. odd $n\geq 3$), we construct a convex $n$-gon $G$ and two convex polygons $A_0, A_1$ contained in $G$, sharing $n$ (resp. $n+1$) common supporting lines such that the configuration fails to satisfy the weak carousel rule, thus showing that the bounds given in Theorem~\ref{thm:carousel_points} are sharp. See Figure~\ref{fig:sharpness} for the constructions with $n=3,4,5,6$.
\begin{figure}[htbp]
  \centering
  \includegraphics[width=\linewidth]{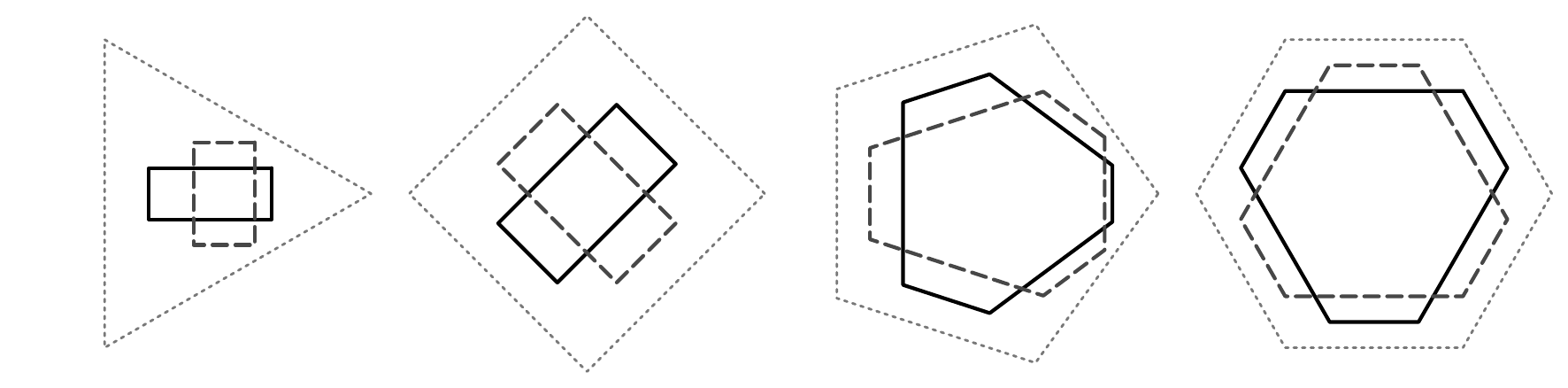}
  \caption{Examples for $n=3,4,5,6$ demonstrating that Theorem~\ref{thm:carousel_points} is sharp. $G$ is dotted. The dashed and solid polygons inside are $A_0$ and $A_1$.}
  \label{fig:sharpness}
\end{figure}

\subsubsection{Sharpness for even $n$} We identify $\R^2$ with the complex plane. Let $G$ be the $n$-gon whose vertices $G_k$ are the $n$th roots of unity, i.e. $G_k = \exp(2\pi ik/n)$ for $k \in \Z/n\Z$. Let $M_k = \frac{1}{2}(G_k+G_{k+1})$ be the midpoint of the edge connecting $G_k$ and $G_{k+1}$. Denote the line segments connecting $M_k$ to $M_{k+1}$ by $L_k$. Define $A_0$ and $A_1$ as follows. Trisect each edge $L_k$ to obtain the points 
\[x_k := \frac{2}{3}M_{k+1} + \frac{1}{3}M_k,\quad y_k:= \frac{1}{3}M_{k+1} + \frac{2}{3}M_k\]
and define 
\[a_{0,k} = \begin{cases}
x_k & \text{ $k$ even} \\
y_k & \text{ $k$ odd,} 
\end{cases}\quad a_{1,k} = \begin{cases}
x_k & \text{ $k$ odd} \\
y_k & \text{ $k$ even} 
\end{cases}\]
and let $A_0$ be the $n$-gon whose vertices are $a_{0,k}$, and $A_1$ the $n$-gon whose vertices are $a_{1,k}$. Then $L_k$ is the intersection of $G$ with one of the $n$ common supporting lines of $A_0$ and $A_1$.
\begin{prop}
The configuration $(A_0, A_1, G_0,\dots, G_{n-1})$ constructed above does not satisfy the weak carousel rule.
\begin{proof}
The configuration is invariant under rotations by $2\pi/n$ and swapping of $A_0$ and $A_1$, so it suffices to show that for the vertex $G_0 = (1,0) \in \R^2$, the inclusion 
\begin{equation}\label{eqn:desired_inclusion}
  A_1\subset \ch(A_0, \cup_{k\neq 0} G_k)
\end{equation}
fails to hold. Consider $L_{n-1}$, which is the segment of the rightmost common supporting line joining the midpoints of the edges of $G$ closest to $G_0$. It has constant $x$-coordinate $x^* = \frac{1+\cos(2\pi/n)}{2}$. Because $n\geq 4$, $\cos(2\pi/n)<1$, so $x^* >\cos(2\pi/n)$. It follows that $H:= \cup_{k\neq 0} G_k$ lies strictly to the left of $L_{n-1}$. Thus, any point in the convex hull $\ch(A_0\cup H)$ lying on $L_{n-1}$ is contained in the convex hull of $A_0$ alone. But by construction, $A_0 \cap L_{n-1}$ and $A_1 \cap L_{n-1}$ are distinct singletons. Hence the inclusion in (\ref{eqn:desired_inclusion}) fails.
\end{proof}

\end{prop}
\subsubsection{Sharpness for odd $n$}
We imitate the construction from the even case, adding an extra vertex for the interior polygons $A_0$ and $A_1$. Let $G = \ch(G_0, \dots, G_{n-1})$ be the regular $n$-gon as above, and let $M_k = \frac{1}{2}(G_k+G_{k+1})$ be the midpoints of the edges of $G$. Define the extra point $Q := \frac{M_{n-1}+M_0}{4}+\frac{1}{2}G_0$, and let $P := \ch (Q, M_0, \dots, M_{n-1})$. To construct $A_0$ and $A_1$, we repeat the same procedure. Set 
\[P_0 = Q, \quad P_k = M_{k-1} \text{ for } 1\leq k\leq n\]
where we take the indices modulo $n+1$. Define
\[x_k := \frac{2}{3}P_{k+1} + \frac{1}{3}P_k,\quad y_k:= \frac{1}{3}P_{k+1} + \frac{2}{3}P_k,
\]
\[a_{0,k} := \begin{cases}
x_k & \text{ $k$ even} \\
y_k & \text{ $k$ odd,} 
\end{cases}\quad a_{1,k} := \begin{cases}
x_k & \text{ $k$ odd} \\
y_k & \text{ $k$ even.} 
\end{cases}\]
Let $A_0$ be the $(n+1)$-gon whose vertices are $a_{0,k}$, and $A_1$ the ($n+1)$-gon whose vertices are $a_{1,k}$. To see that this construction does not satisfy the weak carousel rule, notice that every vertex of $A_0$ and $A_1$ away from the split edges $[M_{n-1}, Q]$ and $[Q, M_0]$ is unchanged from the construction in the even case. So it suffices to verify that if we remove $G_0$, neither of the required inclusions holds. This is a direct coordinate computation. 

\subsection{Corollaries} We now list a few corollaries of Theorem~\ref{thm:carousel_points}. First, {\cite[Theorem 3.1]{ADARICHEVA2019726} and \cite[Theorem 1.1]{czedliCounterexamples}}, concerning two disks contained in a triangle, follow immediately since two distinct disks have at most two common supporting lines. Theorem~\ref{thm:carousel_points} also shows that this result holds if we replace the triangle with any convex $n$-gon for $n\geq 3$.

Next, we consider the case of algebraic curves. The question of whether there exists an extension of the weak carousel rule to ellipses was posed to the author by K. Adaricheva. Theorem~\ref{thm:carousel_points} provides an answer:
\begin{cor}\label{cor:algebraic}
  Let $A_0, A_1$ be compact subsets of the plane contained in a convex $n$-gon $G$. For $i \in \{0,1\}$, suppose $\partial\ch(A_i)$ is a smooth compact oval of an irreducible projective plane curve $C_i$ of degree $d_i$ defined over $\R$. Suppose $C_0 \neq C_1$. If $d_0(d_0-1)d_1(d_1-1) < 2{\lceil \frac{n}{2}\rceil}$,
  then this configuration satisfies the weak carousel rule. In particular, two ellipses contained in a pentagon satisfy the weak carousel rule.
  \begin{proof}
    If $\ch(A_0)=\ch(A_1)$ then the result is immediate, so suppose otherwise. Then every common supporting line of $\ch(A_0)$ and $\ch(A_1)$ is tangent to $C_0$ and $C_1$, hence determines a real point on the intersection of the dual curves $C_0^*$ and $C_1^*$. These dual curves $C_i^*$ are distinct and have degrees at most $d_i(d_i-1)$. Hence, by B{\'e}zout's theorem, there are at most $d_0(d_0-1)d_1(d_1-1)$ common supporting lines. Apply Theorem~\ref{thm:carousel_points}. In the case of ellipses, take $d_0 = d_1 = 2$.
  \end{proof}
\end{cor}

Finally, we apply Theorem~\ref{thm:carousel_points} to prove \cite[Theorem 1.1]{Czedli2019}, which concerns the case where $A_0$ is obtained from $A_1$ by a positive homothety or translation. Recall that for a point $p \in \R^2$ and a constant $\lambda \in \R\setminus \{0\}$, the \emph{homothety} with center $p$ and ratio $\lambda$ is the map $\chi_{p,\lambda}: \R^2\to \R^2$ given by $x \mapsto p+\lambda(x-p)$. If $\lambda >0$, we say $\chi_{p,\lambda}$ is a \emph{positive} homothety. Under composition, the maps ($x \mapsto \lambda x+t$ for $\lambda \in \R_{>0}, t\in \R^2$) form a group containing all positive homotheties and translations.
\begin{cor}\label{cor:homothety}
Assume $A_1\subset \R^2$ is nonempty and compact. If $A_0$ is obtained via a positive homothety $\chi_{p,\lambda}$ or translation $x\mapsto x+t$ of $A_1$, then the weak carousel rule is satisfied for any convex polygon $G$ containing $A_0\cup A_1$.
\begin{proof}
  By Remark~\ref{rem:convex}, replace $A_0$ and $A_1$ with their convex hulls. First, we treat the case where $A_0$ and $A_1$ are homothetic. We compute
  \[h_{A_0}(\eta) = h_{\chi_{p,\lambda}(A_1)}(\eta) =  \lb p-\lambda p, \eta\rb + \lambda \sup_{x \in A_1}\lb x, \eta \rb =  (1-\lambda)\lb p, \eta\rb+ \lambda h_{A_1}(\eta)\]
  which yields 
  \[h(\eta) := h_{A_0}(\eta)-h_{A_1}(\eta) = (\lambda -1)(h_{A_1}(\eta) - \lb p,\eta\rb ) = (\lambda-1)h_{A_1 - p}(\eta).\]
  We may assume $\lambda \neq 1$, for otherwise $A_0 = A_1$. It remains to count the zeros of $h$, which are the zeros of $h_{A_1 - p}$. Consider the following cases:
  \begin{enumerate}
    \item $p \in A_1$. If $\lambda<1$, then for any $a \in A_1$, the point $(1-\lambda)p+\lambda a$ lies in $A_1$, hence $A_0\subset A_1$ and the weak carousel rule holds trivially. If $\lambda>1$, then the same argument in reverse shows that $A_1\subset A_0$.
    \item $p \notin A_1$. Then the origin lies outside of $A_1-p$. A zero of $h_{A_1-p}$ then corresponds to a supporting direction of $A_1-p$ whose corresponding common supporting line passes through the origin, of which there are at most two.
  \end{enumerate}
  Since $h$ has at most two zeros, $A_0$ and $A_1$ have at most two common supporting lines, so we can apply Theorem~\ref{thm:carousel_points} to complete the proof in this case.

  If $A_0 = A_1+t$, then $h(\eta) = \sup_{x \in A_1}\lb x+t, \eta\rb - \sup_{x \in A_1}\lb x, \eta\rb  = \lb t, \eta \rb$. Unless $t=0$ (in which case $A_0 = A_1$), its zeros are the two unit vectors perpendicular to $t$. Again, we conclude via Theorem~\ref{thm:carousel_points}.
\end{proof}
\end{cor}

\section{An application to convex geometries}\label{section:convexgeometries}
In this section we use the weak carousel rule to prove a non-representability result for convex geometries. Recall that a \textit{convex geometry} (not to be confused with the field of study) is a pair $(U,\tau)$ where $U$ is a set and $\tau: 2^U\to 2^U$ is a closure operator satisfying $\tau(\varnothing) = \varnothing$ and the so-called \emph{anti-exchange} axiom:
  \begin{quote}
    \centering
\emph{  For all $x,y \in U$, $x\neq y$, and $A\subset U$ with $A = \tau(A)$, \\
  if $x \in \tau(A\cup \{y\})$ and $x,y \notin A$, then $y \notin \tau(A\cup \{x\})$.}
\end{quote} 
\noindent The classic example is obtained by taking $U$ to be a subset of $\R^d$ and $\tau$ to be $\ch(-)\cap U$.

\subsection{Representation of convex geometries}
A \emph{representation} result asserts an isomorphism between a convex geometry and a lattice of convex subsets of Euclidean space. Precisely, given a convex geometry $(U,\tau)$ and a class $\mathcal{C}$ of convex sets in $\R^d$, we say $(U,\tau)$ is \emph{representable by $\mathcal{C}$} if there exists an injection $\phi: U\to \mathcal{C}$ such that for any subset $A\subset U$ and $u \in U$,
\[u \in \tau(A) \text{\emph{ if and only if } } \phi(u) \subset \ch(\cup_{a \in A}\phi(a)).\]
That is, the closure operator $\tau$ corresponds exactly to the convex hull operator. For a formal exposition of convex geometries, see \cite{edelmanjamison}, which unifies various definitions of convex geometries, or Chapter 5 of \cite{AdNa16} for a comprehensive modern survey. 

In \cite[\S 6]{ADARICHEVA2019726}, the authors use the weak carousel rule to give an explicit example of a convex geometry not representable by disks. Their construction is the $n=3$ case of the following:
\begin{con}
Fix $n\geq 3$. Let $U_n = \{a_0, a_1, g_1, \dots, g_n\}$ and let $\tau_n$ be the closure operator defined on a subset $A\subset U_n$ via
\[
  \tau_n(A) = \begin{cases}
  U_n & \text{if }  \{g_1, \dots, g_n\}\subset A \\
  A & \text{otherwise.} 
  \end{cases}\]
\end{con}
It is straightforward to verify that this defines a convex geometry. 
\begin{lem}\label{lem:counterexample_generation}
Fix $n\geq 3$. Let $\mathcal{C}$ be a class of compact convex subsets of $\R^d$. Suppose the weak carousel rule holds for all $(A_0, A_1, G_1, \dots, G_n) \in \mathcal{C}^{n+2}$ where $A_0\cup A_1 \subset G:=\ch(G_1, \dots, G_n)$. Then the convex geometry $(U_n,\tau_n)$ cannot be represented by $\mathcal{C}$.
\begin{proof}
  Suppose there exists a representation $\phi: U_n\to \mathcal{C}$. Let $A_i := \phi(a_i)$ and $G_j := \phi(g_j)$. Since $\tau_n(\{g_1, \dots, g_n\}) = U_n$, we know that $A_0, A_1 \subset \ch(G_1, \dots, G_n)$. By the weak carousel rule, this implies that there exist $i \in \{0,1\}$ and $j \in \{1, \dots, n\}$ such that
  \[A_i \subset \ch(A_{1-i}, \cup_{k\neq j}G_k),\]
  which, by definition of a representation, holds if and only if
  \[a_i \in \tau_n(\{a_{1-i}, g_1, \dots, \widehat{g_j}, \dots, g_n\}) = \{a_{1-i}, g_1, \dots, \widehat{g_j}, \dots, g_n\},\]
  a contradiction.
\end{proof}
\end{lem}
In particular, it is shown in \cite[Theorem 5.2]{ADARICHEVA2019726} that when $\mathcal{C}$ is the class of disks in $\R^2$ and $n=3$, the assumptions of Lemma~\ref{lem:counterexample_generation} hold. In general, to generate new non-representability results of this nature, it is sufficient to find classes $\mathcal{C}$ of convex shapes such that every tuple $(A_0, A_1, G_1, \dots, G_n) \in \mathcal{C}^{n+2}$  with $A_0, A_1\subset \ch(G_1, \dots, G_n)$ satisfies the weak carousel rule.

\subsection{Positive homothets of convex bodies} 
We now prove Theorem~\ref{thm:nonrep_exmp}, which implies that the class of positive homothets of a fixed convex body cannot represent the previously constructed family of convex geometries $\{(U_n,\tau_n)\}_{n\geq 3}$. For the remainder of this subsection, fix a convex body $K\subset \R^2$, and assume without loss of generality that $0 \in \intr(K)$. Let $\mathcal{H}(K) = \{t + \lambda K: t\in \R^2, \lambda \in \R_{>0}\}$ and let $B := \{\eta: h_K(\eta) \leq 1\}$.

First, notice that the support function of a given homothet $A = t_0+\lambda_0 K \in \mathcal{H}(K)$ is given by $h_A(\eta) = \lb t_0,\eta \rb + \lambda_0 h_K(\eta)$. Hence the support function is affine when restricted to $\partial B$. We can then extend $h_A|_{\partial B}$ to the plane by defining
\begin{equation}\label{eqn:def_l}
  l_A(z) := \lb t_0,z \rb  +\lambda_0.
\end{equation}
This yields for each $A \in \mathcal{H}(K)$ an affine map $l_A: \R^2\to \R$, which retains the information about inclusions between elements of $\mathcal{H}(K)$, just as the support function does:
\begin{lem}\label{lem:l_function_properties}
  Let $C, D \in \mathcal{H}(K)$. Then $C\subset D$ if and only if $l_C(z)\leq l_D(z)$ for all $z \in \partial B$. 
  \begin{proof}
    It suffices by Lemma~\ref{lem:support_function_properties}(c) to show that $l_C(z)\leq l_D(z)$ for all $z \in \partial B$ if and only if $h_C(\eta) \leq h_D(\eta)$ for all $\eta \in S^1$. For fixed $\eta \in S^1$, let $z := \eta/h_K(\eta)$. Here $h_K(\eta)$ is strictly positive since $0 \in \intr(K)$. Then 
    \[h_C(\eta) = h_C(zh_K(\eta)) = h_C(z)h_K(\eta) = l_C(z)h_K(\eta)\]
    since $h_K(z) = 1$. Hence $h_C(\eta)$ is a positive multiple of $l_C(z)$ and the result follows. Conversely, for fixed $z \in \partial B$, taking $\eta = z/|z|$ yields the same argument.
  \end{proof}
\end{lem}

It also follows from Lemma~\ref{lem:support_function_properties} that, for $C, C_1, \dots, C_n \in \mathcal{H}(K)$, $C\subset \ch(\cup_iC_i)$ if and only if $l_C(z)\leq \max_i l_{C_i}(z)$ for all $z \in \partial B$. Hence it suffices to work with these simpler affine functions $l_A$ on the domain $B$ or its boundary $\partial B$. 
\begin{proof}[Proof of Theorem~\ref{thm:nonrep_exmp}] We separate the proof into three steps.
\paragraph{\emph{Step one: various affine functions}} Let
\[a_i:= l_{A_i},\quad g_j = l_{G_j},\quad m = \frac{a_0+a_1}{2}, \quad d = \frac{a_0-a_1}{2}, \quad f_j =  g_j-m.\]
Suppose that $A_0\not\subset A_1$ and $A_1\not\subset A_0$, for otherwise the result is trivial. Then $d$ takes both positive and negative values on $B$, so the line $L:= \{d(z) = 0\}$ must intersect the interior of $B$; write this intersection $B\cap L$ as $[p,q]$. Since $d$ is affine, the points $p$ and $q$ divide $\partial B$ into two arcs 
\[\Gamma_+:= \{z \in \partial B: d(z)\geq 0\},\quad \Gamma_-:= \{z \in \partial B: d(z)\leq 0\},\]
which we can cover with the convex sets 
\[S_j:= \{z \in B: f_j(z)\geq |d(z)|\}, \quad j \in \{1, \dots, n\}. \]
Indeed, since $A_0, A_1\subset \ch(\cup_jG_j)$, we obtain 
\[\max_j f_j(z) = \max_j g_j(z)-m(z) \geq a_i(z) - m(z) = \begin{cases}
d(z) & i = 0 \\
-d(z) & i= 1.
\end{cases}\]

\paragraph{\emph{Step two: toward a contradiction}} Suppose that the weak carousel rule fails; that is, for every $j \in \{1, \dots, n\}$ we have 
\[A_0\not\subset \ch(A_1, \cup_{k\neq j}G_k)\text{ and } A_1\not\subset \ch(A_0, \cup_{k\neq j}G_k).\]
It follows from Lemma~\ref{lem:l_function_properties} that there exists $x_j \in \partial B$ such that 
\begin{equation}\label{eqn:useful_ineqs}
  a_0(x_j)>a_1(x_j), \quad a_0(x_j) >g_k(x_j)\text{ for $k\neq j$}, \quad a_0(x_j)\leq g_j(x_j).
\end{equation}
From these inequalities we obtain that $x_j$ is contained in the relative interior of the arc $\Gamma^+$, that $x_j \in S_j$, and for $k\neq j$ that $x_j \notin S_k$. Similarly, we can repeat this construction using the other non-containment to obtain a point $y_j \in \partial B$ such that $y_j$ is contained in the relative interior of the arc $\Gamma^-$, that $y_j \in S_j$, and that $y_j \notin S_k$ for $k\neq j$. Since $S_j$ is convex, the segment $[x_j, y_j]$ is contained in $S_j$. Because $x_j\in\Gamma^+$ and $y_j\in\Gamma^-$, the segment $[x_j, y_j]$ must intersect the segment $[p,q]$. We conclude that the intersection
\[S_j\cap [p,q] = \{z \in [p,q]: f_j(z)\geq |d(z)|\} = \{z \in [p,q]: f_j(z)\geq 0\} \]
is nonempty (here the second equality holds because $d=0$ on $[p,q]$). But a nonempty intersection of an interval with the halfplane $\{f_j(z)\geq 0\}$ must contain at least one endpoint, so $p\in S_j$ or $q \in S_j$. 
\vspace{5pt}

\paragraph{\emph{Step three: the contradiction}} Now consider for $j\in \{1, \dots, n\}$ the sets 
\[I_j: = \Gamma_+ \setminus S_j = \{z\in \Gamma_+: f_j(z)-d(z)<0\}.\]
Here $I_j$ must be an interval on $\partial B$, for otherwise $S_j$ contains neither $p$ nor $q$, contradicting the previous discussion. We then have $\cap_j I_j = \Gamma^+ \setminus \cup_j S_j = \varnothing$ by the first step. By Helly's theorem for intervals (see, for example, \cite[\S2.1]{ECKHOFF1993}), there exist distinct $\alpha, \beta \in \{1, \dots, n\}$ such that $I_\alpha\cap I_\beta = \varnothing$, i.e. $\Gamma^+ \subset S_\alpha \cup S_\beta$. Since $n\geq 3$, we can choose $\gamma \in \{1, \dots, n\}\setminus \{\alpha,\beta\}$. Then $x_\gamma \in S_\gamma$ and $x_\gamma \in S_\alpha\cup S_\beta$, a contradiction. This completes the proof.
\end{proof} 
Taking $K$ to be a disk, we immediately obtain \cite[Theorem 5.2]{ADARICHEVA2019726}, since all disks are positively homothetic. 
\begin{cor}\label{cor:nonrep}
  For every $n\geq 3$ and every fixed convex body $K\subset \R^2$, the convex geometry $(U_n,\tau_n)$ is not representable by $\mathcal{H}(K)$.
  \begin{proof}
  Combine Lemma~\ref{lem:counterexample_generation} and Theorem~\ref{thm:nonrep_exmp}.
  \end{proof}
\end{cor}
    
\section{Concluding remarks}\label{section:questions}
We end with a few questions. First, the assumption that all shapes are positive homothets of a fixed convex body prevents the weak carousel rule from yielding more general non-representability results. For example, see Figure~\ref{fig:ellipsecounterex} for a representation of $(U_3,\tau_3)$ by non-homothetic ellipses (see \cite[Figure 23]{Adaricheva2024} for another representation of the same convex geometry, called $CG4$ there, by ellipses).
\begin{figure}[htbp]
    \centering    \includegraphics[width=0.7\linewidth]{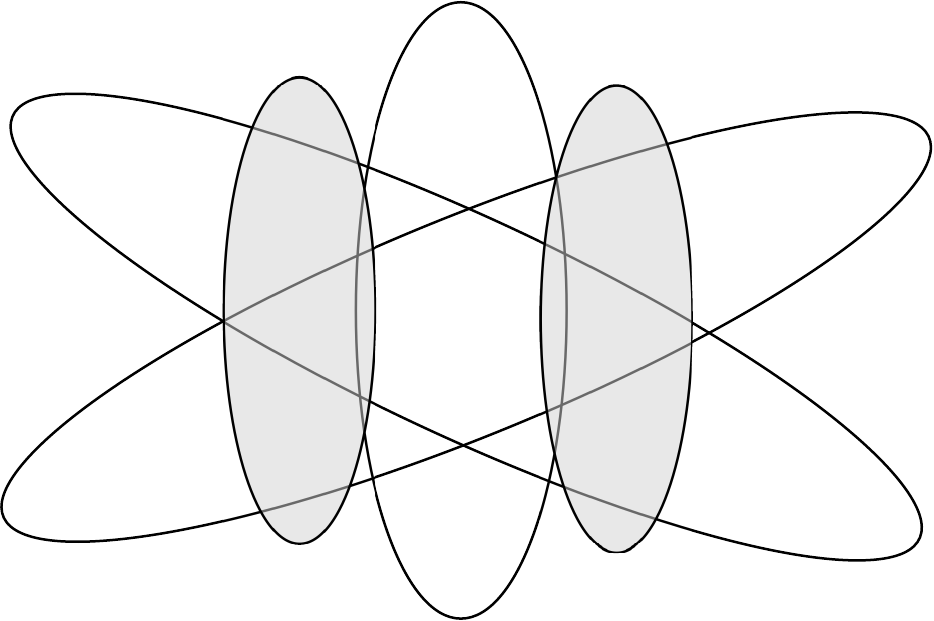}
    \caption{A representation of the convex geometry $(U_3, \tau_3)$ by ellipses. The ellipses $A_0, A_1$ are shaded; the ellipses $G_1, G_2, G_3$ are unshaded. The weak carousel rule fails for $(A_0, A_1, G_1, G_2, G_3)$.}
    \label{fig:ellipsecounterex}
\end{figure}
\begin{quest}
  Can the positive homothety assumption in Theorem~\ref{thm:nonrep_exmp} be removed or weakened to generate more examples of non-representable convex geometries?
\end{quest}
Next, in \cite[Theorem 1.1]{Czedli2017-qi}, Cz{\'e}dli shows that the weak carousel rule characterizes disks in $\R^2$. That is, a compact convex set $A_0 \subset \R^2$ is a disk if and only if for every $A_1$ isometric to $A_0$ and every triangle $G$ containing $A_0\cup A_1$, the weak carousel rule holds. One natural question in this vein, following Corollary~\ref{cor:algebraic}, is whether this property uniquely characterizes a given class of algebraic curves:
\begin{quest}
    Suppose for fixed $k$ and $n$ that the weak carousel rule holds for every arrangement of compact $A_0$ and $A_1$ contained in an $n$-gon $G$, where $\partial\ch(A_0)$ and $\partial\ch(A_1)$ are algebraic curves of degree at most $k$. Does this imply any relation between $k$ and $n$? Is there a converse to Corollary~\ref{cor:algebraic}?
\end{quest}
Lastly, we ask whether some reformulation of the weak carousel rule holds in higher dimensions. Though Lemma~\ref{lem:carousel_general} holds in all dimensions, the constant-sign hypothesis is difficult to verify in dimensions exceeding two. Extensions of Theorem~\ref{thm:carousel_points} using common supporting hyperplanes also seem difficult; as shown in \cite[Example 4.1]{czedliCounterexamples}, the simplest generalization, for an arrangement of two spheres in a tetrahedron in $\R^3$, fails. We ask:

\begin{quest} 
  Is there an analogous weak carousel rule in dimensions $d\geq 3$?
\end{quest}

\section*{Acknowledgments} Part of this research was conducted during the 2025 Baruch College Discrete Mathematics NSF Site REU. I would like to thank Kira Adaricheva for suggesting the project and providing invaluable feedback on numerous drafts, Adam Sheffer for mentorship and advice, and G{\'a}bor Cz{\'e}dli for suggesting interesting questions. I would also like to thank an anonymous referee for helpful comments that greatly improved the paper.

\bibliographystyle{alpha}
\bibliography{mybib}

\newcommand{\etalchar}[1]{$^{#1}$}
\begin{thebibliography}{ADG{\etalchar{+}}24}

\bibitem[AB19]{ADARICHEVA2019726}
Kira Adaricheva and Madina Bolat.
\newblock Representation of convex geometries by circles on the plane.
\newblock {\em Discrete Mathematics}, 342(3):726--746, 2019.

\bibitem[ADG{\etalchar{+}}24]{Adaricheva2024}
Kira Adaricheva, Evan Daisy, Ayush Garg, Grace Ma, Michelle Olson, Cat Raanes, and James Thompson.
\newblock Convex geometries representable with colors, by ellipses on the plane, and impossible by circles.
\newblock {\em Acta Scientiarum Mathematicarum}, 90(1-2):269--322, March 2024.

\bibitem[AN16]{AdNa16}
Kira Adaricheva and J.~B. Nation.
\newblock Convex geometries.
\newblock In George Gr{\"a}tzer and Friedrich Wehrung, editors, {\em Lattice Theory: Special Topics and Applications}, pages 153--179. Springer International Publishing, 2016.

\bibitem[AW10]{Adaricheva2010}
Kira Adaricheva and Marcel Wild.
\newblock Realization of abstract convex geometries by point configurations.
\newblock {\em European Journal of Combinatorics}, 31(1):379--400, 2010.

\bibitem[CK17]{CzedliKincses2017}
Gábor Czédli and János Kincses.
\newblock Representing convex geometries by almost-circles.
\newblock {\em Acta Scientiarum Mathematicarum}, 83(3-4):393--414, December 2017.

\bibitem[CK19]{Czedli2019}
G{\'a}bor Cz{\'e}dli and {\'A}rp{\'a}d Kurusa.
\newblock A convex combinatorial property of compact sets in the plane and its roots in lattice theory.
\newblock {\em Categories and General Algebraic Structures with Applications}, 11:57--92, 2019.

\bibitem[Cz{\'e}14]{CZEDLI201461}
G{\'a}bor Cz{\'e}dli.
\newblock Finite convex geometries of circles.
\newblock {\em Discrete Mathematics}, 330:61--75, 2014.

\bibitem[Cz{\'e}17a]{Czedli2017-qi}
G{\'a}bor Cz{\'e}dli.
\newblock Characterizing circles by a convex combinatorial property.
\newblock {\em Acta Scientiarum Mathematicarum}, 83(3--4):683--701, 2017.

\bibitem[Cz{\'e}17b]{czedliCounterexamples}
G{\'a}bor Cz{\'e}dli.
\newblock An easy way to a theorem of {Kira Adaricheva} and {Madina Bolat} on convexity and circles.
\newblock {\em Acta Scientiarum Mathematicarum}, 83(3--4):703--712, 2017.

\bibitem[Eck93]{ECKHOFF1993}
J\"{u}rgen Eckhoff.
\newblock Helly, {Radon}, and {Carathéodory} type theorems.
\newblock In {\em Handbook of Convex Geometry}, pages 389--448. Elsevier, 1993.

\bibitem[EJ85]{edelmanjamison}
Paul~H. Edelman and Robert~E. Jamison.
\newblock The theory of convex geometries.
\newblock {\em Geometriae Dedicata}, 19(3):247--270, 1985.

\bibitem[Kin17]{Kincses2017}
J{\'a}nos Kincses.
\newblock On the representation of finite convex geometries with convex sets.
\newblock {\em Acta Scientiarum Mathematicarum}, 83(1--2):301--312, 2017.

\bibitem[KNO05]{KASHIWABARA2005129}
Kenji Kashiwabara, Masataka Nakamura, and Yoshio Okamoto.
\newblock The affine representation theorem for abstract convex geometries.
\newblock {\em Computational Geometry}, 30(2):129--144, 2005.
\newblock Special issue on the 19th European Workshop on Computational Geometry.

\bibitem[RR17]{RICHTER20171059}
Michael Richter and Luke~G. Rogers.
\newblock Embedding convex geometries and a bound on convex dimension.
\newblock {\em Discrete Mathematics}, 340(5):1059--1063, 2017.

\bibitem[Sch13]{Schneider_2013}
Rolf Schneider.
\newblock {\em Convex Bodies: The {Brunn--Minkowski} Theory}.
\newblock Encyclopedia of Mathematics and its Applications. Cambridge University Press, Cambridge, 2nd edition, 2013.

\end{thebibliography}
    
\end{document}